\documentclass[12 pt]{article}
\usepackage{amsmath}
\usepackage{amsthm}
\usepackage{graphicx}
\usepackage{bbm,amssymb}
\usepackage[hyperfootnotes=false,colorlinks]{hyperref}
\usepackage{url}

\def\th{\theta}

\def\d{\delta}
\def\Dl{\Delta}
\def\ep{\epsilon}
 
\def\Im{{\rm{Im}}\,} 
\def\a{\alpha}

\def\b{\beta}

\def\G{\Gamma}
\def\g{\gamma}

\def\all{\forall}
\def\xst{\exists}
\def\bksl{\backslash}

\def\equivto{\Longleftrightarrow}
\def\lang{\left\langle}
\def\goto{\rightarrow}
\def\oo{\infty}

\def\rang{\right\rangle}
\def\tld{\tilde}

\def\cd{\cdot}
\def\de{\partial}

\def\inv{^{-1}}
\def\ran{\operatorname{ran}}

\def\spn{\operatorname{span}}

\newcommand{\rc}[1]{\frac{1}{#1}} 

\newcommand{\R}{{\mathbb{R}}}
\newcommand{\C}{{\mathbb{C}}}
\newcommand{\Z}{{\mathbb{Z}}}
\newcommand{\N}{{\mathbb{N}}}

\def\H{{\mathbb{H}}}

\def\Cn{\C^n}

\def\Zn{\Z^n}

\newcommand{\Legendre}[2]{{\left(\frac{#1}{#2}\right)}}

\def\bx{{\mathbf{x}}}
\def\bxi{{\mathbf{\xi}}}
\def\ba{{\mathbf{a}}}
\def\bh{{\mathbf{h}}}
\def\bm{{\mathbf{m}}}
\newtheorem{defi}{Definition}

\newtheorem{lem}{Lemma}

\newtheorem{rem}{Remark}

\newtheorem{con}{Conjecture}
\newtheorem{thm}{Theorem}
\newtheorem{ex}{Example}

\title{Gauss Sphere Problem with Polynomials}
\author{Fan Zheng \footnote{MIT, supported by Summer Program of Undergraduate Research}
}

\begin{document}

\maketitle
\renewcommand{\thefootnote}{}
Mathematics Subject Classifications: Primary 11P21; Secondary 11L03, 11L07
\renewcommand{\thefootnote}{\arabic{footnote}}

\maketitle
\renewcommand{\thefootnote}{}
\begin{center}

Mentor: Chenjie Fan

Project suggested by David Jerison
\renewcommand{\thefootnote}{\arabic{footnote}}
\end{center}

\begin{abstract}
This paper provides an estimate of the sum of a homogeneous polynomial $P$ of degree $\nu$ and mean zero over the lattice points inside a sphere of radius $R$. It is proved that
\[
\sum_{\bx \in \Z^3 \atop |\bx| \le R} P(\bx) = O_{\ep,P}(R^{\nu + 83/64 + \ep})
\]

\end{abstract}

\pagebreak
\tableofcontents

\section{Introduction: A Historical Review}
The Gauss Sphere Problem is a generalization of the famous Gauss Circle Problem to three dimensions. It asks about the number of lattice points in a sphere of radius $R$. While it is easy to see that the leading term is asymptotically $4\pi R^3/3$, estimation of the error term is still open. Compared to the two dimensional case, the Gauss Sphere Problem has received less attention, yet it bears relation with a variety of topics in analytic number theory: average class numbers of negative discriminants, estimates of $L$-functions (see the introduction of \cite{CI}) and Fourier coefficients of modular forms (see \cite{CCU}), to name a few.

The trivial observation, already known to Gauss, that the error arises only from a shell of constant thickness on the surface of the ball, provides an error term of $CR^2$, where $C$ is an effective constant.

The first breakthrough was made by Van der Corput, who used Poisson summation formula to transform the lattice in physical space to the one in frequency space. In this way, counting lattice points translates to summing the Fourier transform of $\chi_{B(R)}$, the characteristic function of the ball of radius $R$. Although the sum itself diverges, it can be brought convergent by smoothing out $\chi_{B(R)}$, or equivalently multiplication of a cut-off function in the frequency space. The width $H$ of the range to be smoothed, and correspondingly the radius $H\inv$ in the frequency space to be summed, is a parameter that can be adjusted to obtain an optimal error bound. Using trivial estimates, Van der Corput obtained an error $CRH\inv$ for the sum in the frequency space, and an error $CR^2H$ from the smoothing in the physical space. Balancing these two errors by setting $H = R^{-1/2}$ an error of $CR^{3/2}$ for the Gauss Sphere Problem was reached.

The basic structure of Van der Corput's argument has remained unchanged ever since. Subsequent improvements came from a more careful analysis of both error terms. In the following discussions we shall use Vinogradov's notations:

\begin{defi}
$f = O(g) \equivto f \ll g \equivto g \gg f \equivto \xst C > 0$ such that $f \le Cg$. $f \asymp g \equivto f \ll g$ and $g \ll f$. Subscripted variables attached to these symbols mean that the implicit constant depends on the variables.
\end{defi}

In addition we introduce a short hand for the exponential function.

\begin{defi}
We define $e(z) = \exp(2\pi iz)$ for $z \in \C$.
\end{defi}

The frequency side was first exploited for improvements, because the Fourier transform of $\chi_{B(r)}$ turns out to be (up to some constants)
\[
R\frac{e(R|\xi|)}{|\xi|^2}
\]
Since the denominator can be removed by Abel summation, the summation in the frequency space essentially involves the partial sum of the exponentials
\[
\sum_{\xi \in \Z^3 \atop |\xi| \le N} e(R|\xi|)
\]
Of the numerous ways invented to deal with such exponential sums, the two most important ones are Van der Corput $A$ and $B$ processes. The $A$ process is also known as Weyl differencing,. The $B$ process is Poisson summation and stationary phase, much in the same spirit as discussed above. These two processes have been abstracted to give the method of exponent pairs. For more details the reader is referred to Appendix B and \cite{GK}.

Improvements in the frequency space culminated in the works of Chen \cite{Ch} and Vinogradov \cite{Vi}, in which they independently improved the estimate of the exponential sum to $O_{\ep}(RH^{-1/2 + \ep})$ in a sufficiently large range so that it can be balanced with $O(R^2H)$ to give an error estimate of $O_\ep(R^{4/3 + \ep})$.

The results of Chen and Vinogradov stood for another thirty years before Chamizo and Iwaniec turned their attention to the physical space. Using character sums, they improved the trivial bound $O(R^2H)$ on the error caused by smoothing to $O_\ep(R^{15/8 + \ep} H^{7/8})$, and thus lowered the 4/3 in the exponent further down to 29/22 \cite{CI}. Currently, the world record on this problem is $O_\ep(R^{21/16 + \ep})$, obtained by Heath-Brown \cite{HB} in much the same way as \cite{CI}. His step forward is an improved error bound $O_\ep(R^{11/6 + \ep} H^{5/6})$ of the character sum.

Last but not least, it should be mensioned that current techniques are still insufficient to prove the famous conjecture that the true error bound should be $O_\ep(R^{1 + \ep})$.

\section{The Method of Chamizo and Iwaniec: Extensions, Improvements and Limitations}
Chamizo and Iwaniec's method has gained popularity in a number of related problems in recent years. Aside from the above-mentioned \cite{HB}, the reader is referred to \cite{CC} and \cite{CCU} for more examples, and to \cite{C} for a non-technical account. This paper provides another application of this method, along with some improvements. The problem considered here is summing a homogeneous polynomial of zero mean on the sphere. In other words, we provide an estimate for
\begin{equation}\label{sum}
\sum_{\bx \in \Z^3 \atop |\bx| \le R} P(\bx)
\end{equation}
where $P$ is a homogeneous polynomial in three variables such that
\[
\int_{|\bx|=R} P(\bx) = 0
\]
It is well-known (see Corollary 2.50 of \cite{Fo}, for example) that $P$ can be written as a linear combination of spherical harmonics times powers of $|\bx|^2$. The zero mean of $P$ ensures that the spherical harmonics involved are all non-constant. Therefore, it suffices to consider only harmonic homogeneous polynomial of degree $\nu > 0$.

Since $P$ has zero mean, there is no main term of the form $cR^{\nu+3}$ as in the Gauss Sphere Problem. Thus the natural form of the estimate is
\[
\sum_{\bx \in \Z^3 \atop |\bx| \le R} P(\bx) = O_{\ep,P}(R^{\nu + \th_\nu + \ep})
\]
where we have taken into account of the fact that $\sup_{|\bx|=R} P(\bx) \asymp R^\nu$.

Normally, we would expect that (by some abuse of notation) $\th_\nu = \th_0$. In other words, the error in estimating the sum of a homogeneous polynomial of degree $\nu$ scales according to the degree of the polynomial. For example, Van der Corput's estimate
\[
\sum_{\bx \in \Z^3 \atop |\bx| \le R} 1 = \frac{4\pi R^3}{3} + O(R^{3/2})
\]
easily generalizes to (see the proof of Lemma 3.5 (c) in \cite{J})
\[
\sum_{\bx \in \Z^3 \atop |\bx| \le R} P(\bx) = O_P(R^{\nu+3/2})
\]

Naively, we would expect Heath-Brown's record-keeping result to give an error of $O_{\ep,P}(R^{\nu + 21/16 + \ep})$ for our problem. While this is true, the exponent 21/16 can actually be lowered to 83/64 if we examine Heath-Brown's method more carefully and try to adapt it to the current case.

As mentioned in the previous section, Heath-Brown decomposed the sum over lattice points into two parts, the ``long sum" and the ``short sum". The long sum is essentially a smoothed version of the original sum:
\[
S_f(R) = \sum_{\bx \in \Z^3} \rc{|\bx|} f(|\bx|)
\]
where $f$ is a cutoff function growing like $|\bx|$ when $0 \le |\bx| \le R$ and decreasing linearly to 0 when $R \le |\bx| \le R + H$, and $H$ is the ``width" of the cutoff function $f$, a parameter tunable for optimal bounds.

On the other hand, the short sum is
\[
S_f(R, H) = \sum_{\bx \in \Z^3 \atop R \le |\bx| \le R+H} \rc{|\bx|} f(|\bx|)
\]
The actual sum (\ref{sum}), in the special case $P = 1$, is then the difference between $S_f(R)$ and $S_f(R, H)$, which are estimated in different ways.

The estimation of the long sum $S_f(R)$ essentially follows \cite{CI}. More precisely, in the notation in \cite{GK}, the operation $ABAB$ is performed to $S_f(R)$ to obtain the following estimate (thanks to \cite{CCU} for pointing out a misprint in \cite{CI})
\begin{equation}\label{CIS_fR}
S_f(R) = \frac{4\pi R^3}{3} + 2\pi HR^2 + O_\ep((RH^{-1/2} + R^{11/8}H^{1/8} + R^{21/16})H^{-\ep})
\end{equation}

The short sum, on the other hand, is converted to a character sum, an idea dated back to Gauss (see (1.2) of \cite{CI}). Its estimation is made by Chamizo and Iwaniec, and improved by Heath-Brown to (see (3) in \cite{HB})
\[
S_f(R, H) = 2\pi R^2H + O_\ep((R^{11/6}H^{5/6} + R^{19/15} + R^{7/6}H^{-1/6})R^{\ep})
\]
Balancing $S_f(R)$ and $S_f(R, H)$ by setting $H = R^{-5/8}$, we obtain Heath-Brown's bound
\[
\sum_{\bx \in \Z^3 \atop |\bx| \le R} 1 = \frac{4\pi R^3}{3} + O_\ep(R^{21/16+\ep})
\]
where, as Heath-Brown has remarked, the exponent 21/16 comes from trading the $H$ factors in the term $RH^{-1/2}$ in $S_f(R)$ and the term $R^{11/6}H^{5/6}$ in $S_f(R, H)$. The term $R^{21/16}$ in $S_f(R)$, on the other hand, is not optimal and has some room for improvement.

The approach taken in this paper is essentially the same as that in \cite{CI} and \cite{HB}, but two differences should be remarked, the first one more fundamental than the second.

The first difference is in the estimate of the short sum. While character sum is used in the case $P = 1$, in our case when $P$ is a harmonic homogeneous polynomial of degree $\nu > 0$, modular forms come into play. In more detail, let
\[
a_n = \sum_{\bx \in \Z^3 \atop |\bx|^2 = n} P(\bx)
\]
and
\[
\th(z) = \sum_{n \in \N} a_n e(nz)
\]
It is known (Proposition 2.1 of \cite{Sh}, or Example 2, P 14 of \cite{Sa}) that $\th$ is a cusp form of weight $k = \nu + 3/2$ with respect to the congruence group
\[
\G_0(4) = \left\{
\begin{pmatrix}
a & b\\
c & d
\end{pmatrix} \in SL_2(\Z), 4|c\right\}
\]
(when $P = 1$, $\th$ is a modular form, but not a cusp form.) Therefore some powerful estimates of the Fourier coefficients $a_n$ of cusp forms are available, of which we use Blomer-Harcos' bound (Corollary 2 of \cite{BH})
\[
a_n \ll_{\ep,\th} n^{k/2 - 5/16 + \ep} (n, 2^\oo)^{5/8} \asymp R^{\nu + 7/8 + \ep} (n, 2^\oo)^{5/8}
\]
and the triangle inequality to obtain Theorem \ref{S_fPRH}:
\begin{equation}\label{shortsum}
S_{f,P}(R,H) = O_{\ep,P}(R^{\nu+\ep} (R^{15/8}H + R))
\end{equation}
which wins a factor of $R^{1/8}$ over the trivial short sum estimate $O_f(R^{\nu+2}H)$. (The GCD term, which means the largest power-of-two factor of $n$, is absorbed in the process of summation.)

The second difference is a slight improvement of the estimate (\ref{CIS_fR}), which is now necessary because the previously-mentioned improvements on the short sum pushes $H$ up, making $R^{21/16}$ dominate $RH^{-1/2}$ in the long sum, so removing the $R^{21/16}$ is necessary (and possible, as already hinted by Heath-Brown). This is achieved by optimizing the two Weyl differencing steps ($A$ processes). Specifically, in \cite{CI}, the lengths of the first and the second Weyl differecing are set to $N^{1/2-\ep}$ and $U$. This, however, is not optimal in all cases, especially when $N \asymp R^{6/5}$, for which the off-diagonal term actually dominates the diagonal term. This paper improves on this by tuning the lengths of the two Weyl differencing steps ($Y$ in Lemma \ref{V_NQR} and $T$ in Lemma \ref{V_NUyR}), if possible, to balance the diagonal and off-diagonal terms. Our new estimate is Theorem \ref{S_fPR}, which for any homogeneous polynomial $P$ of degree $\nu$ having zero mean on the sphere, says:
\begin{equation}\label{longsum}
S_{f,P}(R) = O_{\ep,P}(R^\nu(RH^{-1/2} + R^{17/14} H^{-1/7})H^{-\ep})
\end{equation}

Now let's see the effect of balancing (\ref{shortsum}) and (\ref{longsum}). Let's ignore the common factor $R^\nu$ and  any factor of the form $R^\ep$ or $H^{-\ep}$. Moreover, let's assume that in the short sum the term $R^{15/8}H$ dominates (the other term actually corresponds to the conjecture $\th_\nu = 1$), and that in the long sum the term $R^{17/14} H^{-1/7}$ dominates (which is the very reason we want to optimize it). We obtain
\[
R^{17/14} H^{-1/7} = R^{15/8} H
\]
which gives $H = R^{-37/64}$ and
\[
\sum_{\bx \in \Z^3 \atop |\bx| \le R} P(\bx) = S_{f,P}(R) - S_{f,P}(R, H) = O_{\ep,P}(R^{\nu + 83/64 + \ep})
\]
The second term in $S_{f,P}(R,H)$ is clearly dominated, while the first term in $S_{f,P}(R)$ is
\[
O_{\ep,P}(RH^{-1/2}) = O_{\ep,P}(R^{\nu + 165/128 + \ep})
\]
which is also dominated. Therefore we have succeeded in showing $\th_\nu \le 1 + 19/64$, pushing Heath-Brown's exponent (for the purpose of our problem) down by 1/64, or 5\% in relative terms. In Appendix B we use a recent result of Huxley \cite{Hu} to sketch a proof that $\th_\nu$ can be further reduced to $1 + 35765/121336$.

Finally let's see how far we could possibly go from here. Assuming Lindel\"of Exponent Pair Conjecture \cite{CC}, we can substitute $k = \ep$ and $l = 1/2 + \ep$ in (\ref{exppair}) in Appendix B to reach $\th_\nu \le 1 + 7/24$, coming from the term $RH^{-1/2}$ in the long sum, which can be traced back to the diagonal term in the first Weyl differencing step. This is probably the best one can expect of the long sum, unless it is redone from scratch in a radically different way.

On the other hand, improvements on the short sum may have a larger impact. If we assume Ramanujan's conjecture on modular forms of half integral weight, i.e. (ignoring all the $\ep$'s in the exponent)
\[
|a_n| \ll_f n^{k/2 - 1/2} \asymp R^{\nu + 1/2}
\]
and use the triangle inequality we get
\[
\sum_{n=R^2}^{(R+H)^2} |a_n| \ll_f R^{\nu + 3/2}H
\]
Balancing this result with Theorem \ref{S_fPR} we obtain a significant better result $\th_\nu \le 1 + 1/4$. If we use a general exponent pair $(k, l)$ provided by Theorem \ref{NewS_fPR} we obtain an error bound of
\[
O_{\ep,P}(R^{7/6} + R^{1 + \frac{3k+3l+1}{10k+10l+6}})
\]
The pair in \cite{Hu} yields $\th_\nu \le 1 + 1409/5790$. Assuming the Lindel\"of conjecture we can get $\th_\nu \le 1 + 5/22$.

Finally we state a conjecture analogous to the famous Gauss Sphere Problem.

\begin{con}\label{GaussConj}
$\th_\nu = 1$. In other words. Suppose $P$ is a homogeneous polynomial of degree $\nu$ in three variables and $P$ has zero mean on the sphere, then
\[
\sum_{\bx \in \Z^3 \atop |\bx| \le R} P(\bx) = O_{\ep,P}(R^{\nu + 1 + \ep})
\]
\end{con}

A tabulated summary of all the proved and conjectured exponents mentioned above can be found in Appendix C.

\section{Converting the Long Sum to the Exponential Sum}
Suppose $P$ is any homogeneous polynomial of degree $\nu > 0$ in three variables, not necessarily having zero mean on the sphere.

\begin{defi}
The ``long sum" refers to the following sum:
\[
S_{f,P}(R) = \sum_{\bx \in \Z^3} g_P(\bx)
\]
where $g_P(\bx) = P(\bx) g(\bx)$, $g(\bx) = f(|\bx|)/|\bx|$, and $f$ is the cutoff function
\[
f(x) = 
\begin{cases}
x,\, x \in [0, R]\\
R(R+H-x)/H,\, x \in [R, R+H]\\
0,\, x \ge R+H
\end{cases}
\]
\end{defi}

\begin{lem}\label{Poissonsum}
\begin{align*}
&S_{f,P}(R)\\
= &\int_{\R^3} P(\bx) g(\bx) d\bx + \sum_{\nu_1 + \nu_2 = \nu} R^{\nu_1} \sum_{\bxi \in \Z^3 \bksl 0} \frac{Q_{\nu_1}(\xi) \sin(2\pi R|\xi| + \frac{\pi \nu_1}{2})}{|\xi|^{3 + \nu_1 + 2\nu_2}}\\
+ &\sum_{\nu_1 + \nu_2 + \nu_3 = \nu} H^{\nu_1} (2R + H)^{\nu_2} \sum_{\bxi \in \Z^3 \bksl 0} \frac{\tld Q_{\nu_1, \nu_2}(\xi) \sin(\pi H|\xi| + \frac{\pi \nu_1}{2}) \cos(\pi(2R + H)|\xi| + \frac{\pi \nu_2}{2})}{|\xi|^{3 + \nu_1 + \nu_2 + 2\nu_3}}
\end{align*}
where $Q_{\nu_1}$ and $\tld Q_{\nu_1,\nu_2}$ are homogeneous polynomials of degree $\nu$.
\end{lem}

\begin{proof}
By Poisson summation formula applied to $g$ (whose validity will be justified later in Theorem \ref{S_fPR})
\[
S_{f,P}(R) = \sum_{\bxi \in \Z^3 } \hat g_P(\bxi)
\]
where the convention of the Fourier transform taken here is
\[
\hat f(\xi) = \int_{\R^3} f(\bx) e(-\xi \cd \bx)
\]
and $e(z)$ is defined to be $\exp(2\pi iz)$ for all $z \in \C$.

If $\xi = 0$, then
\[
\hat g_P(\xi) = \int_{\R^3} g_P(\bx) d\bx
\]
is the integral of the smoothed function $g_P$ over $\R^3$, which contributes to the main term of the estimate of the long sum, and which actually vanishes if $P$ has zero mean on the sphere.

If $\xi \neq 0$, then by (5.2) in \cite{CI}, we have
\[
\hat g(\xi) = \frac{\sin(2\pi R |\xi|)}{2\pi^2 |\xi|^3} - \frac R H \frac{\sin(\pi H|\xi|)}{\pi^2 |\xi|^3} \cos(\pi(2R + H)|\xi|)
\]
so
\[
\hat g_P(\xi) = P\left( \frac{-\de_\xi}{2\pi i} \right) \left( \frac{\sin(2\pi R |\xi|)}{2\pi^2 |\xi|^3} - \frac R H \frac{\sin(\pi H|\xi|) \cos(\pi(2R + H)|\xi|)}{\pi^2 |\xi|^3} \right)
\]

By induction on $\nu$, we can show that
\[
P(\de_\xi) \left( \frac{\sin(2\pi R |\xi|)}{|\xi|^3} \right) = \sum_{\nu_1 + \nu_2 = \nu} R^{\nu_1} \frac{Q_{\nu_1}(\xi) \sin(2\pi R|\xi| + \frac{\pi \nu_1}{2})}{|\xi|^{3 + \nu_1 + 2\nu_2}}
\]
and
\begin{align*}
&P(\de_\xi) \left( \frac{\sin(\pi H|\xi|) \cos(\pi(2R + H)|\xi|)}{|\xi|^3} \right)\\
= &\sum_{\nu_1 + \nu_2 + \nu_3 = \nu} H^{\nu_1} (2R + H)^{\nu_2} \frac{\tld Q_{\nu_1, \nu_2}(\xi) \sin(\pi H|\xi| + \frac{\pi \nu_1}{2}) \cos(\pi(2R + H)|\xi| + \frac{\pi \nu_2}{2})}{|\xi|^{3 + \nu_1 + \nu_2 + 2\nu_3}}
\end{align*}
where $Q_{\nu_1}$ and $\tld Q_{\nu_1,\nu_2}$ are homogeneous polynomials of degree $\nu$.
\end{proof}

Now if we sum over $\xi \in \Z^3 \bksl 0$ and use trigonometric identities, we get sums of the following type
\[
\sum_{\xi \in \Z^3 \bksl 0} \frac{Q(\xi) e(R|\xi|)}{|\xi|^m},\, {\rm or}\, \sum_{\xi \in \Z^3 \bksl 0} \frac{Q(\xi) e((R + H)|\xi|)}{|\xi|^m}
\]
where $Q$ is a homogeneous polynomial of degree $\nu$. By Abel summation, it boils down to estimating sums of the following type
\[
\sum_{\xi \in \Z^3 \atop |\xi|^2 \le N} Q(\xi) e(R|\xi|)
\]

This is the sort of exponential sum that will play a key role in the estimation of the long sum.

\section{Estimating the Exponential Sum}
In this section, $Q$ denotes a homogeneous polynomial of degree $\nu \ge 0$ in three variables, not necessarily having zero mean on the sphere. To justify the convergence of Poisson summation later, we introduce a variable $\bh \in \Z^3$. Note however that all the estimates in this section are uniform in $\bh$, which implies uniform convergence of the partial sums of the Fourier series.

\begin{defi}
The exponential sum we are going to estimate is the following:
\[
V_{N,Q,\bh}(R) = \sum_{\xi \in \Z^3 \atop |\xi|^2 \le N} Q(\xi) e(R|\xi| + \bh \cd \xi)
\]
\end{defi}

\begin{lem}\label{V_NQR}
If $R \ge 1$ then
\[
V_{N,Q,\bh}(R) \ll_{\ep,Q} N^{\nu/2 + \ep} (\min\{N^{3/2}, N^{5/4} + N^{15/14} R^{3/14}\})
\]
\end{lem}
\begin{rem}
The right hand side
\[
\ll N^{\nu/2 + \ep} \cd 
\begin{cases}
N^{3/2},\, 1 \le N \ll R^{1/2}\\
N^{15/14} R^{3/14},\, R^{1/2} \ll N \ll R^{6/5}\\
N^{5/4},\, N \gg R^{6/5}
\end{cases}
\]
\end{rem}
\begin{proof}
Clearly it suffices to prove Lemma \ref{V_NQR} for $Q = a^i b^j c^k$ with $i+j+k = \nu$. The fact that $V_{N,Q,\bh}(R) \ll N^{3/2}$ comes from the trivial bound
\[
V_{N,Q,\bh}(R) \le \sum_{\xi \in \Z^3 \atop |\xi|^2 \le N} |Q(\xi)| \le \#\{\xi \in \Z^3, |\xi|^2 \le N\} \sup_{|\xi|^2 \le N} |Q(\xi)| \ll N^{\nu/2 + 3/2}
\]

To prove the second bound, we transform it into an expression that is the starting place of the estimates in \cite{CI}. By dyadic decomposition
\[
V_{N,Q,\bh}(R) = \sum_{k=1}^{\log N} V^\circ_{2^k,Q,\bh}(R)
\]
where
\[
V^\circ_{N,Q,\bh}(R) = \sum_{\xi \in \Z^3 \atop |\xi|^2 \asymp N} Q(\xi) e(R|\xi| + \bh \cd \xi)
\]
so it suffices to show
\[
V^\circ_{N,Q,\bh}(R) \ll_{\ep,Q} N^{\nu/2 + \ep} (N^{5/4} + N^{15/14} R^{3/14})
\]

For $a^2 + b^2 \le N$, let $\chi_{a,b}$ be a piecewise linear function that is equal to 1 on $\Z \cap [-\sqrt{N - a^2 - b^2}, \sqrt{N - a^2 - b^2}]$, and 0 on the other integers. Then by Theorem 7.3 in \cite{GK}, $\hat \chi_{a,b}(\th)$ is uniformly bounded by $K(\th)$, where $\|K\|_{L^1} \ll \log N \ll_\ep N^\ep$. Therefore, up to the terms where two of $a$, $b$ and $c$ are identical, which sum to $(\sqrt N)^2 O(N^{\nu/2}) = O(N^{\nu/2 + 1})$,
\begin{align*}
V^\circ_{N,Q,\bh}(R) &\ll \sum_{a^2 + b^2 + c^2 \asymp N \atop a,b \ge c} e(R\sqrt{a^2 + b^2 + c^2} + ah_1 + bh_2 + ch_3) a^i b^j c^k\\
&= \sum_{a,b \ge c} a^i b^j e(ah_1 + bh_2) \sum_{|c| \le \sqrt N} e(R\sqrt{a^2 + b^2 + c^2} + ch_3) c^k \int_\R \hat \chi_{a,b}(\th)e(c\th)d\th\\
&\le \sum_{a,b \ge c} (a^2 + b^2)^{(i+j)/2} \left| \int_\R \hat \chi_{a,b}(\th) \sum_{|c| \le \sqrt N} e(R\sqrt{a^2 + b^2 + c^2}) c^k e(c(\th + h_3))d\th \right|\\
&\ll \sum_{0 < n = a^2 + b^2 \asymp N} n^{(i+j)/2} d(n) \int_\R |K(\th)| \left| \sum_{|c| \le \sqrt N} e(R\sqrt{n + c^2}) c^k e(c(\th + h_3)) \right| d\th\\
&\ll_\ep N^{(i+j)/2+\ep} \max_\th \sum_{0 < n \asymp N} \left| \sum_{|c| \le \sqrt N} e(R\sqrt{n + c^2}) c^k e(c(\th + h_3)) \right|
\end{align*}

Now let's apply Weyl differencing: for $1 \le Y \ll \sqrt N$ we have
\begin{align*}
&Y \sum_{|c| \le \sqrt N} e(R\sqrt{n + c^2}) c^k e(c(\th + h_3))\\
\asymp &\sum_{|d| \le Y \atop 2|d} \sum_{c \atop |c+d| \le \sqrt N} e(R\sqrt{n + (c+d)^2}) (c+d)^k e((c+d)(\th + h_3))\\
= &\sum_{|c| \le \sqrt N + Y} \sum_{|d| \le Y \atop {|c+d| \le \sqrt N \atop 2|d}} e(R\sqrt{n + (c+d)^2}) (c+d)^k e((c+d)(\th + h_3))
\end{align*}

By Cauchy Inequality,
\begin{align*}
&Y^2 \left| \sum_{|c| \le \sqrt N} e(R\sqrt{n + c^2}) c^k e(c(\th + h_3)) \right|^2\\
\ll &N^{1/2} \sum_{|c| \le \sqrt N + Y} \left| \sum_{|d| \le Y \atop {|c+d| \le \sqrt N \atop 2|d}} e(R\sqrt{n + (c+d)^2}) (c+d)^k e((c+d)(\th + h_3)) \right|^2\\
= &N^{1/2} \sum_{|c| \le \sqrt N + Y} \sum_{|d_{1,2}| \le Y \atop {|c+d_{1,2}| \le \sqrt N \atop 2|d_{1,2}}} e(R(\sqrt{n + (c+d_1)^2} - \sqrt{n + (c+d_2)^2}))\\
&\cd (c+d_1)^k (c+d_2)^k e((d_1-d_2)(\th + h_3))\\
= &N^{1/2} \sum_{d = (d_1-d_2)/2 \atop |d| \le  Y} \sum_{m = c+(d_1+d_2)/2 \atop |m| \le \sqrt N - d} (Y - d) e(R(\sqrt{n + (m + d)^2} - \sqrt{n + (m - d)^2}))\\
&(m+d)^k (m-d)^k e(2d(\th + h_3))\\
\le &N^{k+1}Y + N^{1/2} \sum_{1 \le |d| \le Y} \sum_{|m| \le \sqrt N - d} (Y - d) e(R(\sqrt{n + (m + d)^2} - \sqrt{n + (m - d)^2}))\\
&(m+d)^k (m-d)^k e(2d(\th + h_3))
\end{align*}
Therefore
\begin{align*}
&\left| \sum_{0 < n \asymp N} \sum_{|c| \le \sqrt N} e(R\sqrt{n + c^2}) c^k e(c(\th + h_3)) \right|^2\\
\ll &N \max_\th \sum_{0 < n \asymp N} \left| \sum_{|c| \le \sqrt N} e(R\sqrt{n + c^2}) c^k e(c(\th + h_3)) \right|^2\\
\ll &N^k \left( N^3Y\inv + N^{3/2}Y\inv \sum_{|d| \le Y} \sum_{|m| \le \sqrt N} \left| \sum_{0 < n \asymp N} e(R(\sqrt{n + (m + d)^2} - \sqrt{n + (m - d)^2})) \right| \right)\\
\ll &N^k \left( N^3Y\inv + N^{3/2}Y\inv \sum_{y = 4md \atop y \ll Y\sqrt N} d(y) \left| \sum_{0 < x = n+(m-d)^2 \asymp N} e(f(x, y)) \right| \right)\\
\ll_\ep &N^{k + \ep} (N^3Y\inv + N^{3/2}Y\inv V_{N, Y\sqrt N}(R))
\end{align*}
where $f(x, y) = R(\sqrt{x + y} - \sqrt x)$ and
\[
V_{N,D}(R) = \sum_{y \asymp D} \left| \sum_{0 < x \asymp N} e(f(x, y)) \right|
\]

Hence we conclude that, for all $1 \le Y \ll \sqrt N$,
\begin{align*}
|V^\circ_{N,Q,\bh}(R)|^2 &\ll_{\ep,Q} N^{\nu+\ep} (N^3Y\inv + N^{3/2}Y\inv V_{N, Y\sqrt N}(R))\\
&\ll_\ep N^{\nu+2\ep} (N^3Y\inv + R^{1/2} N^2 Y^{1/6})
\end{align*}
The last step follows from Lemma \ref{sumxy}. As a result
\[
V^\circ_{N,Q,\bh}(R) \ll_{\ep,Q} N^{\nu/2+\ep} (N^{3/2}Y^{-1/2} + R^{1/4} N Y^{1/12})
\]
If $N \gg R^{6/5}$, then we let $Y \asymp \sqrt N$ to get
\[
V^\circ_{N,Q,\bh}(R) \ll_{\ep,Q} N^{\nu/2+\ep} (N^{5/4} + R^{1/4} N^{25/24}) \ll N^{\nu/2+\ep} N^{5/4}
\]
If $R^{1/2} \le N \ll R^{6/5}$, then we let $1 \le Y \asymp N^{6/7} R^{-3/7} \ll \sqrt N$ to get
\[
V^\circ_{N,Q,\bh}(R) \ll_{\ep,Q} N^{\nu/2+\ep} N^{15/14} R^{3/14}
\]
If $N \le R^{1/2}$, then $N^{15/14} R^{3/14} \ge N^{3/2}$ and Lemma \ref{V_NQR} reduces to the trivial bound.
\end{proof}

\begin{lem}\label{sumxy}
If $R \ge 1$ and $1 \le D \ll N$, then
\[
V_{N,D}(R) \ll (\log N) (N^{3/2} + R^{1/2} D^{7/6} N^{-1/12})
\]
\end{lem}

\begin{proof}
By dyadic decomposition
\[
V_{N,D}(R) = \sum_{k=1}^{\log D} V^\circ_{N,2^k}(R) \ll (\log N) \max_{1 \ll D_1 \ll D} V^\circ_{N,D_1}(R)
\]
where
\[
V^\circ_{N,D}(R) = \sum_{y \asymp D} \left| \sum_{0 < x \asymp N} e(f(x, y)) \right|
\]

If $D \ll D_0 = N^{3/2} R\inv$, then
\[
f_x(x, y) = \frac R 2 \left( \rc{\sqrt x} - \rc{\sqrt{x + y}} \right) = \frac{Ry}{2\sqrt x \sqrt{x + y} (\sqrt x + \sqrt{x + y})} \ll 1
\]
Therefore by Theorem 2.1 of \cite{GK}
\[
\left| \sum_{0 < x \asymp N} e(f(x, y)) \right| \ll f_x(x, y)\inv \ll N^{3/2} R\inv D\inv
\]
so $V_{N,D}(R) \ll N^{3/2} R\inv \le N^{3/2}$.

Now suppose $D_0 \ll D \ll N$. This is possible only when $N \ll R^2$. By applying the $B$ process of the one-dimensional Van der Corput's method (Poisson summation and stationary phase, Lemma 3.6 of \cite{GK} with $F = RDN^{-1/2}$) we get
\begin{align*}
\left| \sum_{0 < x \asymp N} e(f(x, y)) \right| &= \sum_{\xi \asymp f_x(x, y)} \frac{e(g(\xi, y) - 1/8)}{|f_{xx}(\a(\xi, y), y)|^{1/2}} + O(\log (RDN^{-3/2}))\\
&+ O(R^{-1/2} D^{-1/2} N^{5/4}) + O(1)\\
&= \sum_{\xi \asymp f_x(x, y)} \frac{e(g(\xi, y) - 1/8)}{|f_{xx}(\a(\xi, y), y)|^{1/2}} + O(\log R) + O(R^{-1/2} D^{-1/2} N^{5/4})
\end{align*}
where $\a(\xi, y)$ satisfies $f_x(\a(\xi, y), y) = \xi$, and
\[
g(\xi, y) = f(\a(\xi, y), y) - \a(\xi, y)\xi
\]

For any $n \in \N$, $\de_x^n f(x, y)$ does not change sign, so (replacing $n$ by $n+1$) it is monotone in $x$. Thus $\a(\xi, y)$ and hence $f_{xx}(\a(\xi, y), y)$ are monotone in $\xi$. By Abel summation, we obtain
\begin{align*}
\left| \sum_{0 < x \asymp N} e(f(x, y)) \right| &\ll \frac{V_{N,U}(y;R)}{\inf_{\xi \asymp U} |f_{xx}(\a(\xi, y), y)|^{1/2}} + R^{-1/2} D^{-1/2} N^{5/4} + \log R\\
&\ll R^{-1/2} D^{-1/2} N^{5/4} V_{N,U}(y;R) + R^{-1/2} D^{-1/2} N^{5/4} + \log R
\end{align*}
for $U \asymp f_x(x, y) = RDN^{-3/2}$, where
\[
V_{N,U}(y;R) = \left| \sum_{\xi \asymp U} e(g(\xi, y)) \right|
\]
Hence
\begin{align*}
V^\circ_{N,D}(R) &= \sum_{y \asymp D} \left| \sum_{0 < x \asymp N} e(f(x, y)) \right|\\
&\ll R^{-1/2} D^{-1/2} N^{5/4} \sum_{y \asymp D} V_{N,U}(y;R) + R^{-1/2} D^{1/2} N^{5/4} + D\log R\\
&\ll R^{-1/2} N^{5/4} \left( \sum_{y \asymp D} |V_{N,U}(y;R)|^2 \right)^{1/2} + R^{-1/2} N^{7/4} + N\log R
\end{align*}

By Lemma \ref{V_NUyR} below we conclude that
\begin{align*}
V^\circ_{N,D}(R) &\ll R^{-1/2} N^{5/4} (R^{1/2} DN^{-3/4} + RD^{7/6} N^{-4/3}) + R^{-1/2} N^{7/4} + N\log R\\
&\ll N^{3/2} + R^{1/2} D^{7/6} N^{-1/12}
\end{align*}
where we have used $1 \le D \ll N \ll R^2$ to dominate the first two terms over the last two.
\end{proof}

\begin{lem}\label{V_NUyR}
If $R \ll N^2$ and $1 \le D \ll N$, then
\[
\sum_{y \asymp D} |V_{N,U}(y;R)|^2 \ll RD^2 N^{-3/2} + R^2 D^{7/3} N^{-8/3}
\]
\end{lem}

\begin{rem}
A more general result utilizing any established exponent pair $(k, l)$ is found in Lemma \ref{NewV_NUyR} in Appendix B. However, the slight improvement obtained by using the exponent pair in \cite{Hu} comes with the price of the inflation of the numerator and denumerator.
\end{rem}

\begin{proof}
Apply Weyl differencing to $V_{N,U}(y;R)$: for some $T \in [1, U]$ to be chosen later,
\begin{align*}
T^2 \sum_{y \asymp D} |V_{N,U}(y;R)|^2 &= \sum_{y \asymp D} \left| \sum_{\xi \asymp U} \sum_{t=0}^{T-1} e(g(\xi + t, y)) \right|^2\\
&\le U \sum_{\xi \asymp U} \sum_{t_1,t_2=0}^{T-1} \sum_{y \asymp D} e(g(\xi + t_1, y) - g(\xi + t_2, y))\\
&= U \sum_{|t| \le T} (T - |t|) \sum_{\xi \asymp U} \sum_{y \asymp D} e(G(\xi, y, t))\\
&\le U^2TD + UT \sum_{1 \le |t| \le T} \sum_{\xi \asymp U} \left| \sum_{y \asymp D} e(G(\xi, y, t)) \right|\\
&\ll U^2TD + U^2T^2 \max_{1 \le |t| \le T} |V_{N,D}(t,\xi;R)|
\end{align*}
where $G(\xi, y, t) = g(\xi + t, y) - g(\xi, y)$, and
\[
V_{N,D}(t,\xi;R) = \sum_{y \asymp D} e(G(\xi, y, t))
\]

From the computations in \cite{CI} (or Corollary 3.5 of \cite{CC}) we know that $\de_y^n G(\xi, y, t) \asymp tND^{-n}$ for $n = 1$, 2, 3, so estimates (2.3.8) and (2.3.9) in \cite{GK} are valid (with $F$ replaced by $tN$ and $N$ replace by $D$). Note that in this case $tN \ge D$, so the first term in both estimates dominates the second term, yielding
\[
|V_{N,D}(t,\xi;R)| \ll \min\{T^{1/2} N^{1/2}, T^{1/6} N^{1/6} D^{1/2}\}
\]
Therefore
\[
\sum_{y \asymp D} |V_{N,U}(y;R)|^2 \ll U^2(T\inv D + \min\{T^{1/2} N^{1/2}, T^{1/6} N^{1/6} D^{1/2}\})
\]

Let's record the numerology of some borderlines. Balancing the first and the second term yields $T_1 = D^{2/3} N^{-1/3}$. Balancing the first and the third yields $T_2 = D^{3/7} N^{-1/7}$. Balancing the second and the third yields $T_3 = D^{3/2} N\inv$. Balancing $T_1$ with $U = RDN^{-3/2} \ge 1$ yields $D_1 = N^{7/2} R^{-3}$. Balancing $T_2$ with $U$ gives $D_2 = N^{19/8} R^{-7/4}$. Balancing $T_3$ with $U$ gives $D_3 = R^2N\inv$. Comparing $D_0 = N^{3/2} R\inv$, $D_1$, $D_2$, $D_3$ and $N$ gives the following list:
\[
\begin{matrix}
N & {\rm ordering}\\
[R^{6/5}, R^2] & D_3 \ll D_0 \ll N, D_2 \ll D_1\\
[R, R^{6/5}] & D_0 \ll D_{1,2,3} \ll N\\
[1, R] & D_1 \ll D_{0,2} \ll N \ll D_3
\end{matrix}
\]

(1) $R^{6/5} \ll N \ll R^2$. We always have $D_3 \ll D_0$ and $N \ll D_1$. Thus $T_{1,3} \gg U$, so we always take the first argument of the min, which is then dominated by $T\inv D$. Therefore by setting $T = U$ we get
\[
\sum_{y \asymp D} |V_{N,U}(y;R)|^2 \ll UD = RD^2 N^{-3/2}
\]

(2) $R \ll N \ll R^{6/5}$.

(2.1) $D_1 \ll D \ll N$. Then we have $T_1 \ll U$, so we balance the first argument of the min with $T\inv D$ by setting $T = T_1 \gg D_1^{2/3} N^{-1/3} = N^2 R^{-2} \gg 1$ to get
\[
\sum_{y \asymp D} |V_{N,U}(y;R)|^2 \ll U^2 D^{1/3} N^{1/3} = R^2 D^{7/3} N^{-8/3}
\] 
(Although the second argument in the min may be smaller, it does not happen for large $D$, so we have to balance the first argument with $T\inv D$ in this case.)

(2.2) $D_0 \ll D \ll D_1$. Then we have $U \ll T_1$, so we take $T = U$ to reach the same conclusion as in (1).

(3) $1 \le N \ll R$.

(3.1) $N^{1/2} \ll D \ll N$. In this case $T_1 \gg 1$, so we can set $T = T_1$ to reach the same conclusion as (2.1).

(3.2) $1 \le D \ll N^{1/2}$. In this case we let $T = 1$:
\[
\sum_{y \asymp D} |V_{N,U}(y;R)|^2 \ll U^2 (D + N^{1/6} D^{1/2}) = R^2 D^3 N^{-3} + R^2 D^{5/2} N^{-17/6}
\]
which is again dominated by (2.1).

Combining (1) through (3) we conclude that for all $D \ll N$,
\[
\sum_{y \asymp D} |V_{N,U}(y;R)|^2 \ll RD^2 N^{-3/2} + R^2 D^{7/3} N^{-8/3}
\]
\end{proof}

\section{Estimating the Long Sum}
\begin{lem}\label{sumxi}
Suppose $Q$ is a homogeneous polynomial of degree $\nu \ge 0$ in three variables, not necessarily having zero mean on the sphere, and $m \ge \nu + 3$, then
\[
\lim_{N \goto \oo} \sum_{0 < |\xi|^2 \le N} \frac{Q(\xi) e(R|\xi| + \bh \cd \xi)}{|\xi|^m} \ll_{\ep,Q} R^\ep
\]
Moreover, the limit on the left hand side converges uniformly (although not absolutely) in $\bh$.
\end{lem}

\begin{proof}
By Abel summation
\begin{align*}
&\sum_{0 < |\xi|^2 \le N} \frac{Q(\xi) e(R|\xi| + \bh \cd \xi)}{|\xi|^m} = \sum_{n=1}^N n^{-m/2} \sum_{|\xi|^2 = n} Q(\xi) e(R|\xi| + \bh \cd \xi)\\
\ll &\sum_{n=1}^N n^{-m/2-1} |V_{n,Q,\bh}(R)| + N^{-m/2} |V_{N,Q,\bh}(R)|
\end{align*}
where
\[
V_{N,Q,\bh}(R) = \sum_{|\xi|^2 \le N} Q(\xi) e(R|\xi| + \bh \cd \xi)
\]
and we have ignored the discrepancy at $\bxi = 0$, which is of order $O_Q(1)$ and is clearly dominated by other terms. By Lemma \ref{V_NQR} we have
\begin{align*}
&\limsup_{N \goto \oo} \left| \sum_{0 < |\xi|^2 \le N} \frac{Q(\xi) e(R|\xi| + \bh \cd \xi)}{|\xi|^m} \right|\\
\ll_{\ep,Q} &\sum_{1 \le n \ll R^{1/2}} n^{-m/2 + \nu/2 + 1/2 + \ep} + R^{3/14} \sum_{R^{6/11} \ll n \ll R^{6/5}} n^{-m/2 + \nu/2 + 1/14 + \ep}\\
+ &\sum_{n \gg R^{6/5}} n^{-m/2 + \nu/2 + 1/4 + \ep} + \lim_{N \goto \oo} N^{-m/2 + \nu/2} (N^{5/4} + N^{15/14} R^{3/14})\\
\ll &n^\ep|_{n=1}^{R^{1/2}} + R^{3/14} n^{-3/7 + \ep}|_{n=R^{6/5}}^{R^{1/2}} + n^{-1/4 + \ep}|_{n=\oo}^{R^{6/5}} \ll R^\ep
\end{align*}
Moreover, for $M > N \gg R^{6/5}$ we have
\begin{align*}
\sum_{N < |\xi|^2 \le M} \frac{Q(\xi) e(R|\xi| + \bh \cd \xi)}{|\xi|^m} &\ll_{\ep,Q} n^{-1/4 + \ep}|_{n=\oo}^{N} +  N^{-m/2 + \nu/2} (N^{5/4} + N^{15/14} R^{3/14})\\
&\goto_{\ep,Q} 0,\, {\rm as}\, N \goto \oo
\end{align*}
\end{proof}

\begin{thm}\label{S_fPR}
Suppose $P$ is a homogeneous polynomial of degree $\nu \ge 0$, not necessarily having zero mean on the sphere, $R \ge 1$ and $H \le 1$, then
\[
S_{f,P}(R) = \int_{\R^3} P(\bx) g(\bx) d\bx + O_{\ep,P} R^\nu H^{-\ep} (RH^{-1/2} + R^{17/14} H^{-1/7})
\]
\end{thm}

\begin{proof}
Consider the function
\[
S_{f,P}(\bh;R) = \int_{\R^3} P(\bx) g(\bx) d\bx + \sum_{\bxi \in \Z^3 \bksl 0} \hat g_P(\bxi) e(\bh \cd \bxi)
\]
where the summation over $\bxi$ in the second term on the right hand side is taken in the sense of Lemma \ref{sumxi}, and $\hat g_P(\bxi)$ (for $\bxi \in \Z^3 \bksl 0$) is given by
\begin{align*}
\hat g_P(\xi) &= \sum_{\nu_1 + \nu_2 = \nu} R^{\nu_1} \frac{Q_{\nu_1}(\xi) e(R|\xi|)}{|\xi|^{3 + \nu_1 + 2\nu_2}}\\
&+ \sum_{\nu_1 + \nu_2 + \nu_3 = \nu} RH^{\nu_1 - 1} (2R + H)^{\nu_2} \frac{\tld Q_{\nu_1, \nu_2}(\xi) \sin(\pi H|\xi| + \frac{\pi \nu_1}{2}) \cos(\pi(2R + H)|\xi| + \frac{\pi \nu_2}{2})}{|\xi|^{3 + \nu_1 + \nu_2 + 2\nu_3}}
\end{align*}

Since $3 + \nu_1 + 2\nu_2 \ge \nu + 3 \ge \deg Q_{\nu_1} + 3$, Lemma \ref{sumxi} applies to show that
\[
\sum_{\bxi \in \Z^3 \bksl 0} \sum_{\nu_1 + \nu_2 = \nu} R^{\nu_1} \frac{Q_{\nu_1}(\xi) e(R|\xi| + \bh \cd \bxi)}{|\xi|^{3 + \nu_1 + 2\nu_2}} \ll_{\ep,Q} R^{\nu + \ep}
\]
Similarly, since $3 + \nu_1 + \nu_2 + 2\nu_3 \ge \deg \tld Q_{\nu_1,\nu_2} + 3$, $H \le 1$, and $2R + H \ll R$,
\begin{align*}
&\sum_{\bxi \in \Z^3 \bksl 0} \sum_{\nu_1 + \nu_2 + \nu_3 = \nu \atop \nu_1 \ge 1} RH^{\nu_1 - 1} (2R + H)^{\nu_2} \frac{\tld Q_{\nu_1, \nu_2}(\xi) \sin(\pi H|\xi| + \frac{\pi \nu_1}{2}) \cos(\pi(2R + H)|\xi| + \frac{\pi \nu_2}{2}) e(\bh \cd \bxi)}{|\xi|^{3 + \nu_1 + \nu_2 + 2\nu_3}}\\
\ll_{\ep,\tld Q} &R^{\nu + \ep}
\end{align*}
(when $\nu = 0$, i.e. $P$ is constant, this sum actually vanishes), so we conclude that
\begin{align*}
S_{f,P,\bh}(R) &= \sum_{\bxi \in \Z^3 \bksl 0} \sum_{\nu_2 + \nu_3 = \nu} RH\inv (2R + H)^{\nu_2} \frac{\tld Q_{0, \nu_2}(\xi) \sin(\pi H|\xi|) \cos(\pi(2R + H)|\xi| + \frac{\pi \nu_2}{2}) e(\bh \cd \bxi)}{|\xi|^{3 + \nu_2 + 2\nu_3}}\\
&+ O_{\ep,P}(R^{\nu + \ep})
\end{align*}
By Abel summation and Lemma \ref{V_NQR} we have
\begin{align*}
&\sum_{0 < |\xi| \le H\inv/2} \frac{\tld Q_{0, \nu_2}(\xi) \sin(\pi H|\xi|) \cos(\pi(2R + H)|\xi| + \frac{\pi \nu_2}{2}) e(\bh \cd \bxi)}{|\xi|^{3 + \nu_2 + 2\nu_3}}\\
= &\sum_{n=1}^{H^{-2}/4} \frac{\sin(\pi H\sqrt n)}{n^{3/2 + \nu_2/2 + \nu_3}} \sum_{\xi \in \Z^3 \atop |\xi|^2 = n} \tld Q_{0, \nu_2}(\xi) \cos\left( \pi(2R + H)|\xi| + \frac{\pi \nu_2}{2} \right) e(\bh \cd \bxi)\\
= &\sum_{N=1}^{H^{-2}/4} \Dl_N \left( \frac{\sin(\pi H\sqrt N)}{N^{3/2 + \nu_2/2 + \nu_3}} \right) \sum_{\xi \in \Z^3\atop 0 < |\xi|^2 \le N} \tld Q_{0, \nu_2}(\xi) \cos\left( \pi(2R + H)|\xi| + \frac{\pi \nu_2}{2} \right) e(\bh \cd \bxi)\\
+ &O(H^{3 + \nu_2 + 2\nu_3} V_{H^{-2}/4,P,\bh}(R))\\
\ll &H\sum_{N=1}^{H^{-2}/4} N^{-2-\nu/2} |V_{N,P,\bh}(R)| + H^{3 + \nu} |V_{H^{-2}/4,P,\bh}(R)|\\
\ll_{\ep,P} &H \sum_{1 \le N \ll H^{-2}} N^\ep(N^{-3/4} + N^{-13/14} R^{3/14}) + H^{1/2} + H^{6/7} R^{3/14}\\
\ll &HN^\ep (N^{1/4} + N^{1/14} R^{3/14})|_1^{H^{-2}} + H^{1/2} + H^{6/7} R^{3/14}\\
= &H^{-2\ep} (H^{1/2} + H^{6/7} R^{3/14})
\end{align*}
where $\Dl_N$ denotes the difference in the variable $N$, and similarly
\begin{align*}
&\sum_{|\xi| > H\inv/2} \frac{\tld Q_{0, \nu_2}(\xi) \sin(\pi H|\xi|) \cos(\pi(2R + H)|\xi| + \frac{\pi \nu_2}{2}) e(\bh \cd \bxi)}{|\xi|^{3 + \nu_2 + 2\nu_3}}\\
= &\sum_{N > H^{-2}/4} \Dl_N \left( \rc{N^{3/2 + \nu_2/2 + \nu_3}} \right) \sum_{\xi \in \Z^3 \atop 0 < |\xi|^2 \le N} \tld Q_{0, \nu_2}(\xi) \sin(\pi H|\xi|) \cos\left( \pi(2R + H)|\xi| + \frac{\pi \nu_2}{2} \right) e(\bh \cd \bxi)\\
\ll &\sum_{N \gg H^{-2}} N^{-5/2-\nu/2} |V_{N,P,\bh}(R)| \ll_{\ep,P} \sum_{N \gg H^{-2}} N^\ep(N^{-5/4} + N^{-10/7} R^{3/14})\\
\ll &N^\ep (N^{-1/4} + N^{-3/7} R^{3/14})|_\oo^{H^{-2}}\\
= &H^{-2\ep} (H^{1/2} + H^{6/7} R^{3/14})
\end{align*}

Combining the two parts we get
\[
S_{f,P}(\bh;R) = \int_{\R^3} P(\bx) g(\bx) d\bx + O_{\ep,P} R^\nu H^{-\ep} (RH^{-1/2} + R^{17/14} H^{-1/7})
\]
By Lemma \ref{sumxi} we know that the sum $S_{f,P}(\bh; R)$ converges uniformly in $\bh$. On the other hand, by Parseval identity we know that the sum converges in $L^2(\R^3/\Z^3)$ to
\[
g_P(\bh) = \sum_{\bx \in \Z^3} P(\bx + \bh) g(\bx + \bh)
\]
Since this is a continuous function, it is identical (everywhere) to the uniform limit $S_{f,P}(\bh; R)$. In particular, taking $\bh = 0$ yields
\begin{align*}
S_{f,P}(R) &= \sum_{\bx \in \Z^3} P(\bx) g(\bx) = S_{f,P}(0; R)\\
&= \int_{\R^3} P(\bx) g(\bx) d\bx + O_{\ep,P} R^\nu H^{-\ep} (RH^{-1/2} + R^{17/14} H^{-1/7})
\end{align*}
\end{proof}

\section{Estimating the Short Sum}
Now we turn to the short sum. Suppose $P$ is a homogeneous polynomial of degree $\nu > 0$ in three variables, now having zero mean on the sphere. Note that $\deg P = 0$ implies $P = 0$, which is trivial.

\begin{defi}
By the ``short sum" we mean the following:
\[
S_{f,P}(R, H) = \sum_{\bx \in \Z^3 \atop R \le |\bx| \le R+H} \rc{|\bx|} f(|\bx|)
\]
Moreover, we let
\[
a_n = \sum_{\bx \in \Z^3 \atop |\bx|^2 = n} P(\bx)
\]
and define the theta-series
\[
\th(z) = \sum_{n \in \N} a_n e(nz) = \sum_{\bx \in \Z^3} P(\bx) e(|\bx|^2 z)
\]
where as usual $e(z) = e^{2\pi iz}$.
\end{defi}

First we suppose that $P$ is harmonic, then by Lemma \ref{modular} in Appendix A, $\th$ is a cusp form of half integral weight $k = \nu + 3/2 \ge 5/2$ for $\G_0(4)$. For all $n \in \N$ we have:
\begin{lem}\label{Sarnak}(\cite{Sa}, Proposition 1.5.5)
\[
a_n \ll_{\ep,\th} n^{k/2 - 1/4 + \ep}
\]
\end{lem}
\begin{lem}\label{Harcos}(\cite{BH}, Corollary 2)
\[
a_n \ll_{\ep,\th} n^{k/2 - 5/16 + \ep} (n, 2^\oo)^{5/8}
\]
\end{lem}

\begin{thm}\label{S_fPRH}
Suppose $P$ is a homogeneou polynomial of degree $\nu \ge 0$ in three variables, having zero mean on the sphere. Then
\[
S_{f,P}(R,H) \ll_{\ep,P} R^{\nu+\ep} (R^{15/8}H + R)
\]
\end{thm}
\begin{proof}
By Corollary 2.50 of \cite{Fo}, we can write
\[
P(\bx) = \sum_{d=0}^{[\nu/2]} |\bx|^{2d} P_{\nu - 2d}(\bx)
\]
where $P_{\nu - 2d}$ is a homogeneous harmonic polynomial of degree $\nu - 2d$.

If $\nu$ is odd, then there is no term $P_0$ in the sum (and actually $a_n$ vanishes identically by considering the symmetry $\bx \goto -\bx$). If $\nu$ is even, then integration over the sphere gives
\[
0 = \int_{S^2} P(\bx) = \sum_{d=0}^{[\nu/2]} \int_{S^2} P_{\nu - 2d}(\bx) = P_0
\]
so the sum does not include $P_0$ either. Now by Lemma \ref{Sarnak} and Lemma \ref{Harcos},
\begin{align*}
|S_{f,P}(R,H)| &= \left| \sum_{n=R^2}^{(R+H)^2} \frac{f(n) a_n}{n} \right| \le \sum_{n=R^2}^{(R+H)^2} |a_n|\\
&\ll_{\ep,P} \sum_{d=0}^{[\nu/2]} R^{2d} \left( R^{\nu -2d + 3/2 - 5/8 + \ep} \sum_{n=R^2 \atop (n, 2^\oo) \ll RH}^{(R+H)^2} (n, 2^\oo)^{5/8} + R^{\nu - 2d + 1 + \ep} \right)\\
&\ll R^{\nu+\ep} \left( R^{7/8} \sum_{k \ll \log RH} \frac{RH}{2^k} 2^{5k/8} + R \right)\\
&\ll R^{\nu+\ep} (R^{15/8}H + R)
\end{align*}
\end{proof}

\section{Appendix A: Modularity of the Theta Function}
Suppose $n \ge 2$. Let $\C[\bx]$ denote a polynomial with complex coefficients in the variable $\bx \in \Cn$. Let $\C[\bx]_\nu$ be the subspace of $\C[\bx]$ consisting of homogeneous polynomials of degree $\nu$. Let $H[\bx]$ denote the subspace of harmonic polynomials.

\begin{lem}\label{decomp}
$\C[\bx]_\nu \cap H[\bx] = \spn_\C\{(\sum_{j=1}^n a_j x_j)^\nu, (a_j)_{j=1}^n \in \Cn, \sum_{j=1}^n a_j^2 = 0\}$.
\end{lem}

\begin{proof}
Define a positive definite inner product on $\C[\bx]$ by letting $\lang \bx^\a, \bx^\b \rang = \a! \d_{\a\b}$. The Laplacian $\Dl$ and the multiplication operator $M(f) = f \cd |\bx|^2$ operate on $\C[\bx]$. It is easy to check that they are adjoints under this inner product. Therefore
\[
H = \ker \Dl = (\ran M)^\perp
\]

We restrict the whole space to $\C[\bx]_\nu$, where the inner product is still positive definite, and denote the right hand side by $L$. It suffices to show that
\[
L = (\ran M)^\perp
\]
Since $\C[\bx]_\nu$ is finite dimensional, it is equivalent to
\[
L^\perp = \ran M
\]

For $f(\bx) = \sum_{|\a|=d} c_\a \bx^\a \in \C[\bx]_\nu$ we have
\[
\lang \left( \sum_{j=1}^n a_j x_j \right)^\nu, f(\bx) \rang = \lang \sum_{|\a|=d} \frac{|\a|!}{\a!} \ba^\a \bx^\a, \sum_{|\a|=d} c_\a \bx^\a \rang = |\a|! \sum_{|\a|=d} c_\a \ba^\a = |\a|! f(\ba)
\]
Therefore
\[
f \in L^\perp \equivto f(a_i) = 0,\, \all (a_i) \in \Cn, \sum_{i=1}^n a_i^2 = 0 \equivto |\bx|^2|f(\bx) \equivto f \in \ran M
\]
where we have used Hilbert's Nullstellensatz and the fact that $(|\bx|^2)$ is a radical ideal in $\C[\bx]$ when $n \ge 2$. (When $n \ge 3$ it is actually irreducible, or equivalently, prime, in the UFD $\C[\bx]$. When $n = 2$ we have $|\bx|^2 = (x_1 + ix_2)(x_1 - ix_2)$ does not have repeated factors.)
\end{proof}

Now suppose $p_\nu \in \C[\bx]_\nu \cap H[\bx]$. Let $\th(z) = \sum_{\bx \in \Zn} p_\nu(\bx) e(|\bx|^2 z)$. Then $\th$ is holomorphic on the upper half plane $\H$.

\begin{lem}\label{modular}
$\th$ is a cusp form of weight $\nu+n/2$ on $\G_0(4) \bksl H$, where
\[
\G_0(N) = \left\{
\begin{pmatrix}
a & b\\
c & d
\end{pmatrix}
\in SL_2(\Z), N|c \right\}
\]
More precisely, for every $\g \in \G_0(4)$ we have
\[
\th(\g z) = j(\g, z)^{2\nu + n} \th(z)
\]
where
\[
j(\g, z) = \Legendre c d \ep_d\inv (cz+d)^{1/2}
\]
where $\Legendre c d$ is the Legendre symbol as extended by Shimura (\cite{Sh}, point 3 in Notation and Terminology), $\ep_d =
\begin{cases}
1, d = 1 \mod 4\\
i, d = 3 \mod 4
\end{cases}$, and all square roots are taken in the principal branch ($\sqrt{-1}$ is taken to be $e(+1/4)$.)
\end{lem}
\begin{proof}
If $\nu$ is odd then $\th = 0$ by considering the symmetry $\bx \goto -\bx$. Therefore we assume $\nu$ to be even.

By Lemma \ref{decomp},
\[
\th \in \spn_{\C} \sum_{\bx \in \Zn} \left( \sum_{j=1}^n a_j x_j \right)^\nu e(|\bx|^2 z),\, \sum_{j=1}^n a_j^2 = 0
\]
Therefore we may assume that $\th$ is one of the basis vectors on the right hand side.

First we suppose $
\begin{pmatrix}
a & b\\
c & d
\end{pmatrix} \in SL_2(\Z)$. If $c = 0$, then $a = d = \pm 1$. In either case we have $j(\g, z) = 1$, consistent with the fact that $\th(z + b) = \th(z)$ for any $b \in \Z$.

Now we suppose $c \neq 0$, then
\begin{align}
\nonumber \th\left( \frac{az+b}{cz+d} \right) &= \th\left( \frac a c - \rc{c(cz+d)} \right)\\
\nonumber &= \sum_{\bm \in (\Z/c\Z)^n} \sum_{\bx \in \Zn} \left( \sum_{j=1}^n a_j (cx_j + m_j) \right)^\nu e\left( \left( \frac a c - \rc{c(cz+d)} \right) |c\bx + \bm|^2\right)\\
\label{thtrans} &= \sum_{\bm \in (\Z/c\Z)^n} e(a\bm^2/c) \sum_{\bx \in \Zn} \left( \sum_{i=1}^n a_j (cx_j + m_j) \right)^\nu e\left( - \frac{c}{cz+d} \left| \bx + \frac \bm c \right|^2\right)
\end{align}

By Poisson summation formula,
\begin{align*}
&\sum_{\bx \in \Zn} \left( \sum_{j=1}^n a_j (cx_j + m_j) \right)^\nu e\left( - \frac{c}{(cz+d)} \left| \bx + \frac \bm c \right|^2\right)\\
= &\sum_{\bxi \in 2\pi\Zn} \exp\left( \frac{i\bm \cd \bxi}{c} \right) \left( \sum_{j=1}^n a_j ci\de_{\xi_j} \right)^\nu \frac{(2\pi)^{n/2}}{\left( \frac{4\pi ic}{cz+d} \right)^{n/2}} e\left( -\rc{2\pi i} \frac{cz+d}{8\pi ic} |\bxi|^2 \right)\\
= &(ci)^\nu \sum_{\bxi \in 2\pi\Zn} \exp\left( \frac{i\bm \cd \bxi}{c} \right) \left( \frac{cz+d}{2ic} \right)^{n/2} \left( \sum_{j=1}^n a_j \de_{\xi_j} \right)^\nu e\left( \rc{16\pi^2} \left( z + \frac{d}{c} \right) |\bxi|^2 \right)
\end{align*}

We compute the derivative on the rightmost exponential inductively.
\[
\left( \sum_{j=1}^n a_j \de_{\xi_j} \right) e\left( \rc{16\pi^2} \left( z + \frac{d}{c} \right) |\bxi|^2 \right)
= C(z, c, d, \ba, \bxi) e\left( \rc{16\pi^2} \left( z + \frac{d}{c} \right) |\bxi|^2 \right)
\]
where
\[
C(z, c, d, \ba, \bxi) = \frac{i}{4\pi} \left( z + \frac{d}{c} \right) \left( \sum_{j=1}^n a_j \xi_j \right)
\]

Suppose
\begin{equation}\label{innerexp}
\left( \sum_{j=1}^n a_j \de_{\xi_j} \right)^k e\left( \rc{16\pi^2} \left( z + \frac{d}{c} \right) \left| \bxi \right|^2\right)
= C(z, c, d, \ba, \bxi)^k e\left( \rc{16\pi^2} \left( z + \frac{d}{c} \right) |\bxi|^2 \right)
\end{equation}
Then
\begin{align*}
\left( \sum_{j=1}^n a_j \de_{\xi_j} \right)^{k+1} e\left( \rc{16\pi^2} \left( z + \frac{d}{c} \right) |\bxi|^2 \right) &= C(z, c, d, \ba, \bxi)^{k+1} e\left( \rc{16\pi^2} \left( z + \frac{d}{c} \right) |\bxi|^2 \right)\\
+ &kC(z, c, d, \ba, \bxi)^{k-1} \sum_{j=1}^n a_j \de_{\xi_j} C(z, c, d, \ba, \bxi)
\end{align*}
Since $\sum_{j=1}^n a_j^2 = 0$, it is easy to see that the second term on the right hand side vanishes, so by induction we have shown that (\ref{innerexp}) holds for all $k \in \N$. Therefore the inner sum of (\ref{thtrans}) is
\begin{align}
\nonumber &(-1)^\nu \sum_{\bxi \in 2\pi\Zn} \exp\left( \frac{i\bm \cd \bxi}{c} \right) \left( \frac{cz+d}{2ic} \right)^{n/2} \left[ \frac{cz+d}{4\pi} \left( \sum_{j=1}^n a_j \xi_j \right) \right]^\nu e\left( \rc{16\pi^2} \left( z + \frac{d}{c} \right) |\bxi|^2 \right)\\
\label{Poissonsummod} = &\frac{(-1)^\nu}{(2i)^{n/2}} \sum_{\bxi \in \Zn} e\left( \frac{\bm \cd \bxi}{c} \right) \left( z + \frac d c \right)^{n/2} \left( \frac{cz+d}{2} \right)^\nu P(\bxi)^\nu e\left( \rc{4} \left( z + \frac{d}{c} \right) |\bxi|^2 \right)
\end{align}

Now we supppose $4|c$ and carry out the summation over $\bm$ in (\ref{thtrans}).
\[
S_n(\bxi, a, c) = \sum_{\bm \in (\Z/c\Z)^n} e\left( \frac{a\bm^2 + \bm \cd \bxi}{c} \right) = \prod_{j=1}^n S(\xi_j, a, c)
\]
where
\[
S(\xi_j, a, c) = \sum_{m \in \Z/c\Z} e\left( \frac{am^2 + m\xi_j}{c} \right) = \sum_{m \in \Z/c\Z} e\left( \frac{d(m^2 + m\xi_j)}{c} \right)
\]

We now show that if $\xi_j$ is odd, then $S(\xi_j, a, c) = 0$. In fact, since $4|c$ and $(c, d) = 1$, $d$ is odd, so
\begin{align*}
2S(\xi_j, a, c) &= \sum_{m \in \Z/c\Z} \left[ e\left( \frac{d(m^2 + m\xi_j)}{c} \right) + e\left( \frac{d((m + c/2)^2 + (m + c/2)\xi_j)}{c} \right) \right]\\
&= \sum_{m \in \Z/c\Z} e\left( \frac{d(m^2 + m\xi_j)}{c} \right) [1 + e(d(m + c/4 + \xi_j/2))] = 0
\end{align*}
Therefore the sum over $\bxi$ in (\ref{thtrans}) is reduced to
\[
\frac{(-1)^\nu}{(2i)^{n/2}} \sum_{\bxi \in \Zn} \left( \prod_{j=1}^n S(2\xi_j, a, c) \right) \left( z + \frac d c \right)^{n/2} (cz+d)^\nu P(\bxi)^\nu e\left( \left( z + \frac{d}{c} \right) |\bxi|^2\right)
\]
where the sum over $\bm$ is reduced to
\[
S(2\xi_j, a, c) = \sum_{m \in \Z/c\Z} e\left( \frac{d(m^2 + 2m\xi_j)}{c} \right) = e\left( -\frac{d\xi_j^2}{c} \right) \sum_{m \in \Z/c\Z} e\left( \frac{dm^2}{c} \right)
\]

By the well-known Gau$\b$ sum,
\[
\sum_{m \in \Z/c\Z} e\left( \frac{dm^2}{c} \right) = 
\begin{cases}
(1 + i)\ep_d\inv \sqrt c \Legendre c d,\, c, d > 0\\
(1 - i)\ep_{-d} \sqrt c \Legendre c {-d},\, c > 0,\, d < 0\\
(1 - i)\ep_d \sqrt{-c} \Legendre {-c} d,\, c < 0,\, d > 0\\
(1 + i)\ep_{-d}\inv \sqrt{-c} \Legendre {-c} {-d},\, c, d < 0
\end{cases}
\]
the identities $\ep_{-d} = i\ep_d\inv$, $1 + i = \sqrt{2i}$,
\[
\begin{cases}
\sqrt{z + \frac d c}\sqrt c = \sqrt{cz + d},\, c > 0\\
\sqrt{z + \frac d c}\sqrt{-c} = i\sqrt{cz + d},\, c < 0
\end{cases}
\]
and the extension of the Legendre symbol by Shimura, the Gau$\b$ sum can be written uniformly as
\[
\sum_{m \in \Z/c\Z} e\left( \frac{dm^2}{c} \right) = 
\sqrt{2i}\ep_d\inv \sqrt{cz+d} \Legendre c d \left( z + \frac d c \right)^{-1/2}
\]
Therefore,
\[
\th\left( \frac{az+b}{cz+d} \right) = (-1)^\nu \ep_d^{-n} (cz+d)^{\nu+n/2} \sum_{\bxi \in \Zn} P(\bxi)^\nu e(z|\bxi|^2)
\]
Recall that $\nu$ is even, so $(-1)^\nu = \ep_d^{-2\nu} = 1$, and
\[
\th\left( \frac{az+b}{cz+d} \right) = \ep_d^{-2\nu-n} (cz+d)^{\nu+n/2} \sum_{\bxi \in \Zn} P(\bxi)^\nu e(z|\bxi|^2) = j(\g, z)^{2\nu + n} \th(z)
\]

It is clear that $\th$ is a cusp form, because the sum in (\ref{Poissonsummod}) is absolutely convergent and decays rapidly as $\Im z \goto +\oo$ for all $\g \in SL_2(\Z)$.
\end{proof}

\section{Appendix B: Better Results with New Exponent Pairs}

Through personal communication, Chamizo pointed out to the author that the exponent 83/64 in the estimate
\[
\sum_{\bx \in \Z^3 \atop |\bx| \le R} P(\bx) = O_{\ep,P}(R^{\nu + 83/64 + \ep})
\]
still has some room for improvement using the new exponent pairs recently obtained by Huxley \cite{Hu}. Exponent pairs are a useful tool in estimating exponential sums of the following type
\[
\sum_{n \asymp N} e(f(n))
\]

More precisely, suppose $f$ behaves sufficiently close to a power function $x^{-s}$, in the sense that there is a constant $c > 0$ such that for all $r \in \N^*$,
\[
\frac{d^r}{dx^r} f(x) \sim c \frac{d^r}{dx^r} x^{-s}
\]
then for every $\ep > 0$, $P > 0$, and $N > N_{\ep,P}$, $f$ is in the class $\mathbf F(N, P, s, c, \ep)$ as defined in P 30, \cite{GK}. Exponent pairs allow us to give estimates to the above exponential sum.

\begin{defi}(P 30, \cite{GK})
For $0 \le k \le 1/2 \le l \le 1$, $(k, l)$ is an exponent pair if we have the following estimate
\[
\sum_{n \asymp N} e(f(n)) \ll_s (cN^{-s-1})^k N^l + c\inv N^{s+1}
\]
\end{defi}
\begin{rem}
If $f' \asymp cN^{-s-1} \gg 1$, then the second term is dominated by the first term.
\end{rem}
\begin{ex}
By the triangle inequality, $(k, l) = (0, 1)$ is an exponent pair.
\end{ex}
Exponent pairs are abstractions of the Van der Corput $A$ and $B$ processes as carried out in the main part of the paper. Specifically, if $(k, l)$ is an exponent pair, then by applying Weyl differencing (and optimization of the length) we can get another exponent pair
\[
A(k, l) = \left( \frac{k}{2k + 2}, \frac{k + l + 1}{2k + 2} \right)
\]
On the other hand, an application of Poisson summation and stationary phase (the $B$-process) will give as the pair
\[
B(k, l) = (l - 1/2, k + 1/2)
\]
For details and proofs we refer the reader to \cite{GK}.

\begin{ex}
From $(k, l) = (0, 1)$ we can get $B(k, l) = (1/2, 1/2)$ and $AB(k, l) = (1/6, 2/3)$, which are basically the pairs we used to bound $V_{N,D}(t,\xi,R)$.
\end{ex}

From Van der Corput $A$ and $B$ processes we can get a host of exponent pairs. This is, however, not the whole story. Bombieri and Iwaniec (\cite{BI}), using large sieve inequalities, obtained a new exponent pair $(k, l) = (9/56 + \ep, 37/56 + \ep)$ which is unreachable from $A$ and $B$ processes. Their results were subsequently improved by Huxley, giving one more exponent pair $(32/205 + \ep, 269/410 + \ep)$ in \cite{Hu}.

New exponent pairs can offer further optimizations. In \cite{CC}, Chamizo and Crist\'obal applied the exponent pair $(k, l) = BA^2(32/205 + \ep, 269 + 410 + \ep)$ to the following sum
\[
V_{N,D}(t,\xi;R) = \sum_{y \asymp D} e(G(\xi, y, t))
\]
which occurs in the proof of Lemma \ref{V_NUyR}. The result is Proposition 3.6 of \cite{CC}, which in our notation says
\begin{lem}\label{NewV_NUyR}
If $(k, l)$ is an exponent pair, $R \ll N^2$ and $1 \le D \ll N$, then
\[
\sum_{y \asymp D} V_{N,U}(y;R) \ll_\ep RD^{\frac{3k + l + 3}{2k + 2}} N^{\frac{-2k - 3}{2k + 2}} + R^{1/2 + \ep} D^{3/2} N^{-3/4}
\]
\end{lem}

Plugging Lemma \ref{NewV_NUyR} into the final lines of Lemma \ref{sumxy} we get
\begin{align*}
V^\circ_{N,D}(R) &\ll R^{-1/2} D^{-1/2} N^{5/4} \sum_{y \asymp D} V_{N,U}(y;R) + R^{-1/2} D^{1/2} N^{5/4} + N\log R\\
&\ll_\ep N^{3/2 + \ep} + R^{1/2} D^{\frac{2k + l + 2}{2k + 2}} N^{\frac{k - 1}{4k + 4}}
\end{align*}
which implies the same estimate for $V_{N,D}(R)$. Finally, backtracing to Lemma \ref{V_NQR}, we get
\begin{equation}\label{NewV_NQR}
|V^\circ_{N,Q,\bh}(R)|^2 \ll_{\ep,Q} N^{\nu + \ep} (N^3 Y\inv + R^{1/2} N^{2 + \frac{k + l - 1}{4k + 4}} Y^{\frac{l}{2k + 2}})
\end{equation}

Balancing the two term gives
\[
Y_0 = N^{\frac{3k - l + 5}{4k + 2l + 4}} R^{-\frac{k + 1}{2k + l + 2}}
\]
Comparing $Y_0$ with $\sqrt N$ gives
\[
N_0 = R^{\frac{2k + 2}{k - 2l + 3}}
\]
Therefore when $N \gg N_0$, $Y_0 \gg \sqrt N$, so we take $Y_0 \asymp \sqrt N$ and the first term in (\ref{NewV_NQR}) dominates. Otherwise, when $N \ll N_0$, we can balance the two terms in (\ref{NewV_NQR}) by setting $Y = Y_0$ to get
\[
|V^\circ_{N,Q,\bh}(R)|^2 \ll_{\ep,Q} N^{\nu + \ep} R^{\frac{k + 1}{2k + l + 2}} N^{2 + \frac{k + 3l + 1}{4k + 2l + 4}}
\]
Combining the two estimates and adding up dyadic intervals we conclude
\[
|V_{N,Q,\bh}(R)| \ll_{\ep,Q} N^{\nu/2 + \ep} (N^{5/4} + R^{\frac{k + 1}{4k + 2l + 4}} N^{1 + \frac{k + 3l + 1}{8k + 4l + 8}})
\]
which improves Theorem \ref{S_fPR} to
\begin{thm}\label{NewS_fPR}
If $f$, $P$ and $R$ are as in Theorem \ref{S_fPR}, then
\[
S_{f,P}(R) = \int_{\R^3} P(\bx) g(\bx) d\bx + O_{\ep,P} R^\nu H^{-\ep} (RH^{-1/2} + R^{1 + \frac{k + 1}{4k + 2l + 4}} H^{-\frac{k + 3l - 1}{4k + 2l + 4}})
\]
\end{thm}

Now we balance the error terms in Theorem \ref{NewS_fPR} and Theorem \ref{S_fPRH} using Lemma 2.4 of \cite{GK} to obtain
\begin{equation}\label{exppair}
\sum_{\bx \in \Z^3 \atop |\bx| \le R} P(\bx) = O_{\ep,P}(R^{\nu+1+\ep} (R^{7/24} + R^{\frac{15k + 21l + 1}{40k + 40l + 24}}))
\end{equation}

Plugging in the exponent pair $(k, l) = BA^2(32/205 + \ep, 269 + 410 + \ep) = (743/2024 + \ep, 269/506 + \ep)$ we get the dominant exponent
\[
\nu + 1 + 35765/121336 + \ep
\]
which is a further improvement on the previous exponent $\nu + 83/64 + \ep$. The latter can in turn be recovered by taking the well known exponent pair $(k, l) = (1/2, 1/2) = B(0, 1)$.

\section{Appendix C: Summary of Proved and Conjectured $\th_\nu$}

This section summarizes all the proved and conjectured exponents mentioned in Section 2 (and some more). For simplicity, we have omitted the normalizing factors $R^\nu$, the $O$ symbol and all $\ep$'s in the exponents. References and decimal values are included in parentheses. The ellipsis refers to terms that are clearly dominated by the main term, and question marks indicate conjectures. The last column refers to the optimal value of $\a$ such that setting $H = R^\a$ gives the desired error bound. The reader is referred to Remark \ref{Acronyms} for meaning of the acronyms and to Section 2 for a detailed account of relevant terminology. 
{\small
\[
\begin{matrix}
\rm{Long\ sum\ estimates} & \rm{Short\ sum\ estimates} & \th_\nu & \frac{\log H}{\log R} & \rm{Applicability}\\

RH\inv\ (\rm{Van\ der\ Corput}) & R^2H \text{ (Trivial)} & 3/2 & -1/2 & \all P\\

 & & (1.5) & (-.5)\\

RH^{-1/2} + \dots & R^2H \text{ (Trivial)} & 4/3 & -2/3 & \all P\\

\text{(\cite{Ch} and \cite{Vi})} & & (1.33333) & (-.66667)\\

RH^{-1/2} + R^{21/16} & R^{15/8} H^{7/8} + \dots & 29/22 & -7/11 & P = 1\\

+ R^{11/8} H^{-1/8} \text{ (\cite{CI})} & \text{ (\cite{CI})} & (1.31818) & (-.63636)\\

RH^{-1/2} + R^{21/16} & R^{11/6} H^{5/6} + \dots & 21/16 & -5/8 & P = 1\\

+ R^{11/8} H^{-1/8} \text{ (\cite{CI})} & \text{ (\cite{HB})} & (1.3125) & (-.625)\\

RH^{-1/2} + R^{17/14} H^{-1/7} & R^{15/8}H + R & 83/64 & -37/64 & \int_{S^2} P = 0\\

\text{(Theorem \ref{S_fPR})} & \text{(Theorem \ref{S_fPRH})} & (1.29688) & (-.57813)\\

RH^{-1/2} + R^{\frac{15987}{13220}} H^{\frac{-1947}{13220}} & R^{15/8}H + R & \frac{157101}{121336} & -\frac{17601}{30334} & \int_{S^2} P = 0\\

\text{(Theorem \ref{NewS_fPR} with \cite{Hu})} & \text{(Theorem \ref{S_fPRH})} & (1.29476) & (-.58024)\\

RH^{-1/2} + R^{6/5} H^{-1/10}?? & R^{15/8}H + R & 31/24?? & -7/12?? & \int_{S^2} P = 0\\

\text{(Theorem \ref{NewS_fPR} with LEP)} & \text{(Theorem \ref{S_fPRH})} & (1.29167) & (-.58333)\\

RH^{-1/2} + R^{17/14} H^{-1/7} & R^{3/2} H^{1/2}?? & 23/18?? & -4/9?? & P = 1\\

\text{(Theorem \ref{S_fPR})} & \text{(GLH)} & (1.27778) & (-.44444)\\

RH^{-1/2} + R^{6/5} H^{-1/10}?? & R^{3/2} H^{1/2}?? & 5/4?? & -1/2?? & P = 1\\

\text{(Theorem \ref{NewS_fPR} with LEP)} & \text{(GLH)} & (1.25) & (-.5)\\

RH^{-1/2} + R^{17/14} H^{-1/7} & R^{3/2}H?? & 5/4?? & -1/4?? & \int_{S^2} P = 0\\

\text{(Theorem \ref{S_fPR})} & \text{(RC)} & (1.25) & (-.25)\\

RH^{-1/2} + R^{\frac{1454}{1217}} H^{\frac{-461}{2434}} & R^{3/2}H?? & \frac{7199}{5710}?? & -\frac{743}{2895}?? & \int_{S^2} P = 0\\

\text{(Theorem \ref{NewS_fPR} with \cite{Hu})} & \text{(RC)} & (1.24662) & (-.25338)\\

RH^{-1/2} + R^{6/5} H^{-1/10}?? & R^{3/2}H?? & 27/22?? & -3/11?? & \int_{S^2} P = 0\\

\text{(Theorem \ref{NewS_fPR} with LEP)} & \text{(RC)} & (1.22727) & (-.27273)\\

\text{Ultimate Conjecture} & \text{i.e. Conjecture \ref{GaussConj}:} & 1???? & ???? & \all P
\end{matrix}
\]
}

\begin{rem}\label{Acronyms}
Acronyms:

LEP = Lindel\"of Exponent Pair Conjecutre \cite{CC}.

GLH = Generalized Lindel\"of Hypothesis \cite{CC}.

RC = Ramanujan Conjecture \cite{Sa}.
\end{rem}

\section{Acknowledgements}
The author wish to express his thanks to MIT for offering such an incredible program as SPUR, to Prof. David Jerison for suggesting this marvellous topic, to Prof. Etingof and Prof. Fox for meticulously overseeing the progress of the project, to Chenjie Fan for his very helpful mentoring over the six weeks, and to Prof. Chamizo for enlightening and encouraging correspondence.

Speaking of the technicalities, credits go to Prof. Etingof for opening the door of modular forms to me, to Prof. Fox for suggesting some connections to combinatorics, to Chenjie Fan for pointing the direction to the proof of Lemma \ref{decomp} and justification of the convergence of the Poisson summation formula, and for helping the author with some \LaTeX ing problems. Thanks are also due to Prof. Chamizo for bringing to my attention the latest results in the theory of exponent pairs.
\bibliographystyle{plain}
\bibliography{Gauss}
\end{document}